\RequirePackage{fix-cm}
\documentclass[final,times,3p,twocolumn,leqno]{elsarticlehack}
\usepackage{amsmath, amsfonts, amsthm, amssymb}
\usepackage{placeins}
\usepackage{pdfpages}
\usepackage[hypertexnames=false,colorlinks=true,breaklinks=true,bookmarks=true,urlcolor=blue,citecolor=blue,linkcolor=blue,bookmarksopen=false,draft=false]{hyperref}
\usepackage[ruled,vlined,linesnumbered,norelsize]{algorithm2e}
\usepackage{comment}
\usepackage{lipsum}

\theoremstyle{plain}
\newtheorem{theorem}{Theorem}
\newtheorem{lemma}[theorem]{Lemma}

\newtheorem{proposition}[theorem]{Proposition}

\newtheorem{remark}[theorem]{Remark}

\newtheorem{definition}[theorem]{Definition}

\providecommand{\openbox}{\leavevmode
  \hbox to.77778em{%
  \hfil\vrule
  \vbox to.575em{\hrule width.5em\vfil\hrule}%
  \vrule\hfil}}
\makeatletter
\DeclareRobustCommand{\qed}{%
  \ifmmode
    \eqno \def\@badmath{$$}
    \let\eqno\relax \let\leqno\relax \let\veqno\relax
    \hbox{\openbox}%
  \else
    \leavevmode\unskip\penalty9999 \hbox{}\nobreak\hfill
    \quad\hbox{\openbox}%
  \fi
}
\makeatother

\usepackage{xstring}

\usepackage[utf8]{inputenc}
\usepackage{amsmath,amssymb}
\usepackage{xcolor}
\usepackage{multirow}

\DeclareMathOperator{\rank}{rank}

\makeatletter
\renewcommand*{\top}{%
  {\mathpalette\@transpose{}}%
}
\newcommand*{\@transpose}[2]{%
  \scriptsize
  \raisebox{\depth}{$\m@th#1\mathsf{T}$}%
}
\makeatother

\begin{document}
\renewcommand{\arraystretch}{1.3} %
\begin{frontmatter}

\title{On computing sparse generalized inverses}

\author[label1,label2]{Gabriel Ponte}
\author[label2]{Marcia Fampa}
\author[label1]{Jon Lee}
\author[label3]{Luze Xu}

\affiliation[label1]{organization={University of Michigan},
            city={Ann Arbor},
            country={USA}}

\affiliation[label2]{organization={Federal University of Rio de Janeiro},
            country={Brazil}}
            
\affiliation[label3]{organization={UC Davis},
            country={USA}}

%

\date{\today}%

\begin{abstract}
The well-known M-P (Moore-Penrose) pseudoinverse is used in several linear-algebra applications; for example, to compute least-squares solutions of inconsistent  systems of linear equations.  It is uniquely characterized by four properties, but not all of them need to be satisfied for some applications. For computational reasons, it is convenient then, to construct sparse block-structured matrices satisfying  relevant properties of the M-P pseudoinverse for specific applications. (Vector) 1-norm minimization has been used to induce sparsity  in this context. Aiming at row-sparse generalized inverses motivated by the least-squares application, we consider  2,1-norm minimization (and generalizations). In particular, we show that a 2,1-norm minimizing generalized inverse satisfies two additional  M-P pseudoinverse properties, including the one needed for computing least-squares solutions. We present  
mathematical-optimization formulations
 related to finding row-sparse generalized inverses that can be solved very efficiently,  and compare their solutions numerically 
 to  generalized inverses constructed by other methodologies, also aiming at sparsity and row-sparse structure.   

\end{abstract}


\end{frontmatter}


\section{Introduction}\label{sec:introd}
The well-known M-P  (Moore-Penrose) pseudoinverse, independently discovered  by  E.H. Moore and R. Penrose, is used in several linear-algebra applications --- for example, to compute least-squares solutions of inconsistent  systems of linear equations. If $A=U\Sigma V^\top$ is the real singular-value decomposition of $A$ (see \cite{GVL1996}, for example),
then the
M-P pseudoinverse of $A$ can be defined as $A^{\dagger}:=V\Sigma^{\dagger} U^\top$, where $\Sigma^{\dagger}$
has the shape of the transpose of the diagonal matrix $\Sigma$, and is derived from $\Sigma$
by taking reciprocals of the non-zero (diagonal) elements of $\Sigma$ (i.e., the non-zero
singular values of $A$).
The following theorem gives a fundamental characterization of the M-P pseudoinverse.
\begin{theorem}[\cite{Penrose}]
For $A\in\mathbb{R}^{m \times n}$, the M-P pseudoinverse $A^{\dagger}$ is the unique 
 $H\in\mathbb{R}^{n \times m}$ satisfying:
	\begin{align}
		& AHA = A \label{property1} \tag{P1}\\
		& HAH = H \label{property2} \tag{P2}\\
		& (AH)^{\top} = AH \label{property3} \tag{P3}\\
		& (HA)^{\top} = HA \label{property4} \tag{P4}
	\end{align}
\end{theorem}

Following \cite{RohdeThesis} (also see \cite{BenIsrael1974}), we say that a \emph{generalized inverse} is any $H$ satisfying  \ref{property1}. The property \ref{property1} is particularly important in our context; without it,
the all-zero matrix --- extremely sparse and carrying no information at all about $A$ --- would
satisfy the other three properties.

A generalized inverse is \emph{reflexive} if it satisfies \ref{property2}.  Theorem 3.14 in \cite{RohdeThesis} states that: ($i$) if $H$ is a generalized inverse of $A$, then $\mathrm{rank}(H)\ge\mathrm{rank}(A)$, and ($ii$) a generalized inverse $H$ of $A$ is reflexive if and only if $\mathrm{rank}(H)=\mathrm{rank}(A)$. Therefore, enforcing \ref{property2}
gives us the lowest possible rank of a generalized inverse --- a very desirable property.

Following \cite{XFLPsiam},  we say that $H$ is \emph{ah-symmetric} if it satisfies \ref{property3}. That is, ah-symmetric means that $AH$ is symmetric.
If $H$ is an ah-symmetric generalized inverse, then $\hat{x}:=Hb$ solves $\min\{\|Ax-b\|_2:~x\in\mathbb{R}^n\}$ (see \cite{FFL2016,campbell2009generalized}).
So  not  all of the M-P properties are required for a generalized inverse to solve a key problem.
Even if a given matrix is sparse, its M-P pseudoinverse can be completely dense, often leading to a high computational burden in its applications, especially when we are dealing with high-dimensional matrices. Similarly,  we say that $H$ is \emph{ha-symmetric} if it satisfies \ref{property4}, and ha-symmetric generalized inverses have applications as well. 

It is well known that structured sparsity is more useful than sparsity, because of computational efficiency and for explainability. In the context of the least-squares application of an 
ah-symmetric generalized inverse $H$, it is desirable to have few rows that have non-zeros, as then the associated linear model is more explainable.
Because no generalized inverse of $A$ can have rank less than $A$,  every generalized inverse of $A$ has at least
$\mathrm{rank}(A)$ rows with non-zeros. 
So a \emph{row-sparse} generalized inverse of $A$ 
is one having or nearly having 
$\mathrm{rank}(A)$ rows with non-zeros. 
Similarly, we can speak of \emph{column-sparse} generalized inverses.

\smallskip
\noindent {\bf Literature.} \cite{dokmanic,dokmanic1,dokmanic2} suggested 1-norm minimization for inducing sparsity in left and right-inverses. \cite{FFL2016} suggested using relaxations based 1-norm minimization
on subsets M-P properties aiming at sparse generalized inverses. 
\cite{FampaLee2018ORL} suggested a polynomial-time local search with a performance guarantee, 
aiming at minimizing the 1-norm over block-structured generalized inverses of various types.
\cite{XFLPsiam} extended this to other generalized inverses. 
\cite{FLPXjogo} followed up the work of \cite{XFLPsiam} with a computational study. 
Following the approach of \cite{XFLPsiam}, 
\cite{XFLrank12} carefully analyzed the theory for the cases of rank-1 and rank-2 matrices. 
\cite{ojmoFLP21} studied trading off rank against sparsity for ah-symmetric generalized-inverses, using various optimization models. 

\smallskip 
\noindent {\bf Overview.} In \S\ref{sec:struc}, we establish some structural results concerning 
generalized inverses of different types.
In \S\ref{sec:row} (respectively, \S\ref{sec:col}), we investigate minimizing a component-wise increasing function
of the vector of 2-norms of the rows (resp., columns) of a generalized inverse. 
We establish that every such minimizer satisfies \ref{property2} and \ref{property3} (resp., \ref{property4}). In \S\ref{sec:searching}, we demonstrate the value of our results in 
 \S\ref{sec:row} by demonstrating how mathematical-optimization formulations
 related to finding row-sparse generalized inverses can be solved rather efficiently. 
 In the interest of space, we do not do this also for column-sparse generalized inverses. 
 In \S\ref{sec:ls}, we review the very-fast local-search procedure of \cite{XFLPsiam}
 for finding a row-sparse ah-symmetric generalized inverse with low 1-norm. 
 In \S\ref{sec:exp}, we provide computational results indicating that  the LP-based methods 
 tested in \cite{FLPXjogo} can be made to scale much better, using our results in \S\ref{sec:searching}. Nevertheless, we can observe that the local-search procedure of \cite{XFLPsiam} continues to scale better and maintains other desirable properties. 
 Still, we can see that LP becomes a more viable option than was previously believed. 

\smallskip 
\noindent {\bf Notation.} We use $\mathbf{e}_i$  for the $i$-th standard unit vector.

\section{Decomposing generalized inverses}\label{sec:struc}

Let $A \in \mathbb{R}^{m \times n}$ with rank $r$ and $A = U \Sigma V^\top$ be the singular-value decomposition of $A$, where $U \in \mathbb{R}^{m \times m}$, $V \in \mathbb{R}^{n \times n}$ are orthogonal matrices  ($U^\top U = I_m, V^\top V= I_n$), and $\Sigma \in \mathbb{R}^{m \times n}$ with \[\Sigma := \begin{bmatrix}\underset{\scriptscriptstyle r\times r}{D} & \underset{\scriptscriptstyle r\times (n-r)}{0}\\[1.5ex]
\underset{\scriptscriptstyle (m-r)\times r}{0} & \underset{\scriptscriptstyle (m-r)\times (n-r)}{0}\end{bmatrix}~,\]
with $D$ being a diagonal matrix with rank $r$. 
Let $H \in \mathbb{R}^{n \times m}$ and $V^\top H U = \Gamma$, where
\[\Gamma:= \begin{bmatrix}\underset{\scriptscriptstyle r\times r}{X} & \underset{\scriptscriptstyle r\times (m-r)}{Y}\\[1.5ex]
\underset{\scriptscriptstyle (n-r)\times r}{Z} & \underset{\scriptscriptstyle (n-r)\times (m-r)}{W}\end{bmatrix}~,\]
then $H=I_n H I_m = (VV^\top) H (U^\top U) = V\Gamma U^\top$.

Considering the notation above, we present in the following, results related to the M-P pseudoinverse  properties that will be used to prove our main results in the next section.  

\begin{lemma}\label{lem:P1} {\bf (Structure of generalized inverses).}
 \ref{property1} is equivalent to $H = V\Gamma U^\top$ with $X = D^{-1}$.
\end{lemma}
\begin{proof}
\begin{align*}
    &AHA = A\Leftrightarrow\\
    &U\Sigma V^\top V\Gamma U^\top U\Sigma V^\top=U\Sigma V^\top\Leftrightarrow\\
    &U\Sigma \Gamma \Sigma V^\top = U \Sigma V^\top\Leftrightarrow\\
    &U^\top U\Sigma \Gamma \Sigma V^\top V = U^\top U \Sigma V^\top V\Leftrightarrow\\
    &\Sigma \Gamma \Sigma = \Sigma\Leftrightarrow\\
    &\begin{bmatrix}
        D & 0\\
        0 & 0
        \end{bmatrix}\begin{bmatrix}
        X & Y\\
        Z & W
    \end{bmatrix}
    \begin{bmatrix}
        D & 0\\
        0 & 0
    \end{bmatrix} = \begin{bmatrix}
        D & 0\\
        0 & 0
    \end{bmatrix}\Leftrightarrow\\
    &\begin{bmatrix}
        DX & DY\\
        0 & 0
    \end{bmatrix}\begin{bmatrix}
        D & 0\\
        0 & 0
    \end{bmatrix} = \begin{bmatrix}
        D & 0\\
        0 & 0
    \end{bmatrix}\Leftrightarrow\\
    &\begin{bmatrix}
        DXD & 0\\
        0 & 0
    \end{bmatrix} = \begin{bmatrix}
        D & 0\\
        0 & 0
    \end{bmatrix}\Leftrightarrow\\
    &D^{-1}DXDD^{-1} = D^{-1}DD^{-1}\Leftrightarrow\\
    &X = D^{-1}.
\end{align*}

\vspace{-0.2in}

\end{proof}

\begin{lemma}\label{lem:P2}  {\bf (Structure of reflexive generalized inverses).}
  If \ref{property1} is satisfied, then \ref{property2} is equivalent to $H = V\Gamma U^\top$, where $ZDY = W$.
\end{lemma}
\begin{proof}
    \begin{align*}
        &HAH = H\Leftrightarrow\\
        &V\Gamma U^\top U \Sigma V^\top V\Gamma U^\top = V\Gamma U^\top\Leftrightarrow\\
        &V\Gamma \Sigma \Gamma U^\top = V\Gamma U^\top\Leftrightarrow\\
        &V^\top V\Gamma \Sigma \Gamma U^\top U = V^\top V\Gamma U^\top U\Leftrightarrow\\
        &\Gamma \Sigma \Gamma = \Gamma\Leftrightarrow\\
        &\begin{bmatrix}
        X & Y\\
        Z & W
    \end{bmatrix}\begin{bmatrix}
        D & 0\\
        0 & 0
    \end{bmatrix}\begin{bmatrix}
        X & Y\\
        Z & W
    \end{bmatrix} = \begin{bmatrix}
        X & Y\\
        Z & W
    \end{bmatrix}\Leftrightarrow\\
    &\begin{bmatrix}
        XD & 0\\
        ZD & 0
    \end{bmatrix}\begin{bmatrix}
        X & Y\\
        Z & W
    \end{bmatrix} = \begin{bmatrix}
        X & Y\\
        Z & W
    \end{bmatrix}\Leftrightarrow\\
    &\begin{bmatrix}
        XDX & XDY\\
        ZDX & ZDY
    \end{bmatrix} = \begin{bmatrix}
        X & Y\\
        Z & W
    \end{bmatrix}\Leftrightarrow\\
    &\begin{bmatrix}
        X & Y\\
        Z & ZDY
    \end{bmatrix} = \begin{bmatrix}
        X & Y\\
        Z & W
    \end{bmatrix}\Leftrightarrow\\
    &ZDY = W.
    \end{align*}

    \vspace{-0.2in}
    
\end{proof}

\begin{lemma}\label{lem:P3}  {\bf (Structure of ah-symmetric generalized inverses).}
    If \ref{property1} is satisfied, then \ref{property3} is equivalent to $H = V\Gamma U^\top$, where $Y = 0$.
\end{lemma}
\begin{proof}
\begin{align*}
    &AH = (AH)^\top\Leftrightarrow\\
    &U\Sigma V^\top V \Gamma U^\top = (U\Sigma V^\top V \Gamma U^\top)^\top\Leftrightarrow\\
    &U\Sigma \Gamma U^\top = (U\Sigma \Gamma U^\top)^\top\Leftrightarrow\\
    &U\Sigma \Gamma U^\top = U \Gamma^\top \Sigma^\top U^\top\Leftrightarrow\\
    &U^\top U\Sigma \Gamma U^\top U = U^\top U \Gamma^\top \Sigma^\top U^\top U\Leftrightarrow\\
    &\Sigma \Gamma = \Gamma^\top \Sigma^\top \Leftrightarrow\\
    &\begin{bmatrix}
        D & 0\\
        0 & 0
    \end{bmatrix}\begin{bmatrix}
        X & Y\\
        Z & W
    \end{bmatrix} = \begin{bmatrix}
        X^\top  & Z^\top\\
        Y^\top & W^\top
    \end{bmatrix}\begin{bmatrix}
        D^\top & 0\\
        0 & 0
    \end{bmatrix}\Leftrightarrow\\
    &\begin{bmatrix}
        DX & DY\\
        0 & 0
    \end{bmatrix} = \begin{bmatrix}
        X^\top D^\top  & 0\\
        Y^\top D^\top & 0
    \end{bmatrix}\Leftrightarrow\\
    &\begin{bmatrix}
        I_r & DY\\
        0 & 0
    \end{bmatrix} = \begin{bmatrix}
        I_r  & 0\\
        (DY)^\top & 0
    \end{bmatrix}\Leftrightarrow\\
    &Y = 0.
\end{align*}

\vspace{-0.2in}

\end{proof}

\begin{lemma}\label{lem:P4}  {\bf (Structure of ha-symmetric generalized inverses).}
    If \ref{property1} is satisfied, then \ref{property4} is equivalent to $H = V\Gamma U^\top$, where $Z = 0$.
\end{lemma}
\begin{proof}
\begin{align*}
    &HA = (HA)^\top\Leftrightarrow\\     
    &V\Gamma U^\top U \Sigma V^\top = (V\Gamma U^\top U \Sigma V^\top)^\top\Leftrightarrow\\
    &V \Gamma \Sigma V^\top = V\Sigma^\top \Gamma^\top  V^\top\Leftrightarrow\\ 
    &V^\top V \Gamma \Sigma V^\top V = V^\top V\Sigma^\top \Gamma^\top  V^\top V\Leftrightarrow\\
    &\Gamma \Sigma  = \Sigma^\top \Gamma^\top  \Leftrightarrow\\
    &\begin{bmatrix}
        X & Y\\
        Z & W
    \end{bmatrix}\begin{bmatrix}
        D & 0\\
        0 & 0
    \end{bmatrix} = \begin{bmatrix}
        D^\top & 0\\
        0 & 0
    \end{bmatrix}\begin{bmatrix}
        X^\top & Z^\top\\
        Y^\top & W^\top
    \end{bmatrix}\Leftrightarrow\\
    &\begin{bmatrix}
        XD & 0\\
        ZD & 0
    \end{bmatrix} = \begin{bmatrix}
        D^\top X^\top & D^\top Z^\top\\
        0 & 0
    \end{bmatrix}\Leftrightarrow\\
    &\begin{bmatrix}
        I_r & 0\\
        ZD & 0
    \end{bmatrix} = \begin{bmatrix}
        I_r & (ZD)^\top\\
        0 & 0
    \end{bmatrix}\Leftrightarrow\\
    &Z = 0.
    \end{align*}

    \vspace{-0.2in}
    
\end{proof}


\section{Minimizing functions of row 2-norms}\label{sec:row}

We consider the optimization problem 
\begin{align}\tag{\mbox{$P^{f_{\mbox{\tiny{row}}}}_{1}$}}\label{probf2norm}
\min&~f_{\mbox{\tiny{row}}}(\|H_{1\cdot}\|_2\,,\|H_{2\cdot }\|_2\,,\ldots,\|H_{n\cdot }\|_2)\\
\mbox{s.t.}&~ AHA=A\nonumber,
\end{align}
where $f_{\mbox{\tiny{row}}}:\mathbb{R}_+^n\rightarrow \mathbb{R}$ is increasing in each argument.

\begin{theorem}\label{thm:21row}
Suppose that $H^*=V\Gamma U^\top$ is an optimal solution to \rm{\ref{probf2norm}}, then $Y=0,W=0$. 
\end{theorem}
\begin{proof}
    From Lemma \ref{lem:P1}, \ref{property1} is equivalent to $H = V\Gamma U^\top$ with $X = D^{-1}$.
    Then 
\begin{align*}
    H_{i\cdot} &= \mathbf{e}_i^\top V \begin{bmatrix}
        D^{-1} & Y\\
        Z & W
    \end{bmatrix}U^\top~,
\end{align*}
for $i=1,\ldots,n$. Note that 
\begin{equation}\label{proofth}
    \begin{array}{ll}
    \|H_{i \cdot }\|_{2} &= \left\|\mathbf{e}_i^\top V \begin{bmatrix}
        D^{-1} & Y\\
        Z & W\end{bmatrix}U^\top\right\|_2\\
        &= \left\|\mathbf{e}_i^\top V \begin{bmatrix}
        D^{-1} & Y\\
        Z & W
\end{bmatrix}\right\|_2\\
&= \left\| \begin{bmatrix}
        \mathbf{e}_i^\top V \begin{bmatrix}
            D^{-1}\\
            Z
        \end{bmatrix} &         \mathbf{e}_i^\top V \begin{bmatrix}
            Y\\
            W
        \end{bmatrix}
\end{bmatrix}\right\|_2\\
&= \sqrt{
        \left\|\mathbf{e}_i^\top V \begin{bmatrix}
            D^{-1}\\
            Z
        \end{bmatrix}\right\|_2^2 + \left\|\mathbf{e}_i^\top V \begin{bmatrix}
            Y\\
            W
        \end{bmatrix}\right\|_2^2
}~.
\end{array}
\end{equation}
Note that $Y$ and $W$ are variables in \ref{probf2norm}, and its objective  is to minimize $f_{\mbox{\tiny{row}}}$, which is increasing in each argument $\|H_{i\cdot}\|_2$\,, for $i=1,\ldots,n$.

Then, as
        $\left\|\mathbf{e}_i^\top V \begin{bmatrix}
            Y\\
            W
        \end{bmatrix}\right\|_2^2 \geq 0$, we must have  $\mathbf{e}_i^\top V\begin{bmatrix}
            Y\\
            W
        \end{bmatrix}=0$, for all $i=1\ldots,n$, in  an optimal solution of \rm{\ref{probf2norm}}, i.e., we must have $Y\! = \!0$ and $W\! =\! 0$.
\end{proof}

\begin{remark}
    The proof of Theorem \ref{thm:21row} is not valid for other inner norms in the objective of \rm{\ref{probf2norm}}, because we can only remove the matrix $U$ from the norm in \eqref{proofth} if we have the 2-norm.
\end{remark}

\begin{remark}
    If {\rm$f_{\mbox{\tiny{row}}}$} is non-decreasing instead of increasing in each argument,  we cannot say that  $Y=0$ and $W=0$ at every optimal solution of \rm{\ref{probf2norm}}. However, we still have that $Y=0$ and $W=0$ at some optimal solution.  
\end{remark}

\begin{proposition}
Suppose that $H=V\Gamma U^\top$ is an optimal solution to \rm{\ref{probf2norm}}. Then, 
 \ref{property2} and \ref{property3} are satisfied.
\end{proposition}

\begin{proof}
    As $H$ satisfies \ref{property1},
    we have from Lemma \ref{lem:P1} that $X = D^{-1}$. From Theorem \ref{thm:21row} we show that $Y = 0$, $W = 0$, so from Lemmas \ref{lem:P2} and \ref{lem:P3} we have that 
    \ref{property2} and \ref{property3} are satisfied.
\end{proof}

\section{Minimizing functions of column 2-norms}\label{sec:col}

We consider the optimization problem 
\begin{align}\tag{\mbox{$P^{f_{\mbox{\tiny{col}}}}_{1}$}}\label{probf2normcol}
\min&~f_{\mbox{\tiny{col}}}(\|H_{\cdot 1}\|_2\,,\|H_{\cdot 2 }\|_2\,,\ldots,\|H_{\cdot m}\|_2)\\
\mbox{s.t.}&
\quad  AHA=A\nonumber,
\end{align}
where $f_{\mbox{\tiny{col}}}:\mathbb{R}_+^m\rightarrow \mathbb{R}$ is increasing in each argument. 


\begin{theorem}\label{thm:21col}
Suppose $H^*=V\Gamma U^\top$ is an optimal solution to {\rm \ref{probf2normcol}}, then $Z=0,W=0$.
\end{theorem}
\begin{proof}
    From Lemma \ref{lem:P1}, \ref{property1} is equivalent to $H = V\Gamma U^\top$ with $X = D^{-1}$.
    Then, 
\begin{align*}
    H_{\cdot i } &=  V \begin{bmatrix}
        D^{-1} & Y\\
        Z & W
    \end{bmatrix}U^\top \mathbf{e}_i~,
\end{align*}
for  $i=1,\ldots,m$. Note that 
\begin{align*}
    \|H_{\cdot i}\|_{2} \!&= \left\|V \begin{bmatrix}
        D^{-1} & Y\\
        Z & W\end{bmatrix}U^\top \mathbf{e}_i \right\|_2\\
        &= \left\| \begin{bmatrix}
        D^{-1} & Y\\
        Z & W
\end{bmatrix} U ^\top \mathbf{e}_i\right\|_2\\
&= \left\| \begin{bmatrix}
        \begin{bmatrix}
            D^{-1} & Y
        \end{bmatrix}U ^\top \mathbf{e}_i\\     \begin{bmatrix}
            Z &   W
        \end{bmatrix}U ^\top \mathbf{e}_i
\end{bmatrix}\right\|_2\\
&= \sqrt{
        \left\| \begin{bmatrix}
            D^{-1} & Y
        \end{bmatrix}U ^\top \mathbf{e}_i\right\|_2^2 \!+\! \left\| \begin{bmatrix}
            Z &  W
        \end{bmatrix}U ^\top \mathbf{e}_i\right\|_2^2\,.
}
\end{align*}

Note that $Z$ and $W$ are variables in \ref{probf2normcol} and its objective  is to minimize $f_{\mbox{\tiny{col}}}$, which is increasing in each argument $\|H_{\cdot i}\|_2$, for $i=1,\ldots,m$. 
Then, as
$\left\| \begin{bmatrix}
            Z &  W
            \end{bmatrix} U^\top \mathbf{e}_i\right\|_2^2 \!\geq\! 0$, we must have  that $\begin{bmatrix}
            Z & W
    \end{bmatrix}U^\top \mathbf{e}_i\!=\!0$, for all $i=1,\ldots,m$, in an optimal solution of \ref{probf2normcol},  i.e., we must have  $Z = 0$ and $W = 0$.
\end{proof}

\begin{proposition}
Suppose $H=V\Gamma U^\top$ is an optimal solution to {\rm \ref{probf2normcol}}. Then 
\ref{property2} and \ref{property4} are satisfied.
\end{proposition}

\begin{proof}
     As $H$ satisfies \ref{property1},
     we have from Lemma \ref{lem:P1} that $X = D^{-1}$. From Theorem \ref{thm:21col} we show that $Z = 0$, $W = 0$, so from Lemmas \ref{lem:P2} and \ref{lem:P4} we note that \ref{property2} and \ref{property4} are satisfied.
\end{proof}

\section{Searching for sparse ah-symmetric reflexive generalized inverses}\label{sec:searching}


Next, aiming at sparse and block-structured ah-symmetric reflexive generalized inverses, we consider a special case of \ref{probf2norm}, where $f_{\mbox{\tiny{row}}}$ is the 1-norm and, therefore, we minimize the 2,1-norm of $H$.  

\begin{align*}
\tag{\mbox{$P_1^{2,1}$}}
\label{prob:C}
z(P_1^{2,1})
:=\min&\; \|H\|_{2,1}\\
\mbox{s.t.}&~  AHA=A.
\end{align*}

Minimizing the 1-norm of a matrix has been widely applied in the literature as a surrogate to minimizing the 0-norm, or the sparsity.  So we also consider the two  following alternative formulations to obtain $H$.

\begin{align*}\
\tag{\mbox{$(P_1^{2,1})^1$}}
\label{prob:A}
z((P_1^{2,1})^1)
:=\min&\;\|H\|_{1}\\
\mbox{s.t.}& ~ AHA=A,\\
&\; \|H\|_{2,1}\leq 
z(P_1^{2,1}).
\end{align*}

\begin{align*}
\tag{\mbox{$P^1_{123}$}}
\label{prob:B}
z(P^1_{123})
:=\min&\; \|H\|_{1}\\
\mbox{s.t.}& ~ AHA=A,\\
&\; HAH=A,\\
&\; AH=(AH)^\top.
\end{align*}




From \S\ref{sec:struc}, we have that the solution $H \in \mathbb{R}^{n \times m}$ of the three formulations \ref{prob:C}, \ref{prob:A} and \ref{prob:B}\,, can be written as $H=V\Gamma U^\top$, where
\[\Gamma:= \begin{bmatrix}\underset{\scriptscriptstyle r\times r}{D^{-1}} & \underset{\scriptscriptstyle r\times (m-r)}{0}\\[1.5ex]
\underset{\scriptscriptstyle (n-r)\times r}{Z} & \underset{\scriptscriptstyle (n-r)\times (m-r)}{0}\end{bmatrix}.\]
Let 
\[V:= \begin{bmatrix}\underset{\scriptscriptstyle n\times r}{V_1} & \underset{\scriptscriptstyle n\times (n-r)}{V_2}\end{bmatrix}~, \, U:= \begin{bmatrix}\underset{\scriptscriptstyle m\times r}{U_1} & \underset{\scriptscriptstyle m\times (m-r)}{U_2}\end{bmatrix}~, \, G:=V_1D^{-1}U_1^\top\,.
\]

Therefore, formulations \ref{prob:C}, \ref{prob:A} and \ref{prob:B} can be reformulated \emph{much more tractably} as, respectively,

\begin{equation}\tag{\mbox{$\mathcal{P}_1^{2,1}$}}\label{prob:barC}
z(\mathcal{P}_1^{2,1}):=\min\; \sum_{i=1}^n 
\left\|\mathbf{e}_i^\top V \begin{bmatrix}
            D^{-1}\\
            Z
        \end{bmatrix}\right\|_2
\end{equation}

\vspace{0.2in}

\begin{equation}\tag{\mbox{$(\mathcal{P}_1^{2,1})^1$}}\label{prob:barA}
\begin{array}{rrl}
z((\mathcal{P}_1^{2,1})^1):=\!\!\!&\min&\left\| V \begin{bmatrix}
        D^{-1} & 0\\
        Z & 0\end{bmatrix}U^\top\right\|_1\\ [15pt]
&\mbox{s.t.}&\displaystyle\sum_{i=1}^n  \left\|\mathbf{e}_i^\top V \begin{bmatrix}
            D^{-1}\\
            Z
        \end{bmatrix}\right\|_2\leq z(\mathcal{P}_1^{2,1}).
        \\[20pt]
        = &\!\!\min\limits_{{F\in\mathbb{R}^{n\times m},}\atop{Z\in\mathbb{R}^{(n-r)\times r}}} &\sum_{i=1}^{n-r}\sum_{j=1}^r F_{ij}\\[5pt]
        &\mbox{s.t.}& F -  V_2ZU_1^\top\leq G,\\[5pt]
       && F +  V_2ZU_1^\top\geq -G,\\[5pt]
       &&\displaystyle\sum_{i=1}^n  \left\|\mathbf{e}_i^\top V \begin{bmatrix}
            D^{-1}\\
            Z
        \end{bmatrix}\right\|_2\leq z(\mathcal{P}_1^{2,1}).
\end{array}
\end{equation}

\vspace{0.2in}

\begin{equation}\tag{\mbox{$\mathcal{P}^1_{123}$}}\label{prob:barB}
\begin{array}{rrl}
z(\mathcal{P}^1_{123}):=&\min& \left\| V \begin{bmatrix}
        D^{-1} & 0\\
        Z & 0\end{bmatrix}U^\top\right\|_1 \\[15pt]
        = &\min &\left\| G + V_2ZU_1^\top\right\|_1\\
        [10pt]
        = &\min\limits_{{F\in\mathbb{R}^{n\times m},}\atop{Z\in\mathbb{R}^{(n-r)\times r}}} &\sum_{i=1}^{n-r}\sum_{j=1}^r F_{ij}\\[5pt]
        &\mbox{s.t.}& F -  V_2ZU_1^\top\leq G,\\[5pt]
       && F +  V_2ZU_1^\top\geq -G.
\end{array}
\end{equation}

\section{Local-search procedure}\label{sec:ls}
    

\cite{XFLPsiam} devised a local-search procedure, working with
generalized inverses having
minimum row sparsity,
to find  an approximate 1-norm minimizing ah-symmetric generalized inverse.
\cite{FLPXjogo} demonstrated the computational effectiveness of their algorithm, versus the alternative
of solving \ref{prob:B}\,. In light of our new
formulations \ref{prob:barC}, \ref{prob:barA} and \ref{prob:barB},
it is worth conducting new experiments, comparing \cite{XFLPsiam}
and our new formulations. In the remainder of this section, we review the local search of \cite{XFLPsiam}, and in the next section, we present our computational results.

\begin{definition}\label{def:localsearch_ahsym}
For $A\!\in\!\mathbb{R}^{m\times n}$, let $r \!:= \!\rank(A)$. Let $S$ be any ordered subset of $r$ elements from $\{1,\dots,m\}$ such that these $r$ rows of $A$ are linearly independent. For $T$ an ordered subset of $r$ elements from $\{1,\dots,n\}$,    
if   $|\det(A[S,T])|$ cannot be increased by swapping an element of $T$ with one from its complement,  we say that $A[S,T]$ is a 
local maximizer for the absolute determinant on the set of $r\times r$ nonsingular submatrices of $A[S,:]$.
\end{definition}


\begin{theorem}[\cite{XFLPsiam}]\label{thm:ahconstruction}
For $A\!\in\!\mathbb{R}^{m\times n}$, let $r \!:= \!\rank(A)$. For any $T$, an ordered subset of $r$ elements from $\{1,\dots,n\}$, let $\hat{A}\!:=\!A[:,T]$ be the $m \times r$ submatrix of $A$ formed by columns $T$. If $\rank(\hat{A})\!=\!r$, let
$
\hat{H}\! := \!\hat{A}^{\dagger}\! \!=\!(\hat{A}^\top\hat{A})^{-1}\hat{A}^\top.
$
The $n \times m$ matrix $H$ with all rows equal to zero, except rows $T$, which are given by $\hat{H}$, is an ah-symmetric reflexive generalized inverse of $A$.
\end{theorem}

In \cite{XFLPsiam}, a simple local-search procedure was proposed to obtain  a matrix $\tilde{A}:=A[S,T]$ which is  a 
local maximizer for the absolute determinant on the set of $r\times r$ non-singular submatrices of $A[S,:]$, for a given ordered subset $S$ of $r$ elements from $\{1,\dots,m\}$ such that the rows of $A[S,:]$ are linearly independent. Then, an $n \times m$ ah-symmetric reflexive generalized inverse $H$ is constructed over $\hat{A}:=A[:,T]$,   by Theorem \ref{thm:ahconstruction}.
\cite{XFLPsiam} described how to make this algorithm run in polynomial time, and demonstrated that the resulting $H$
(which has optimal row-sparsity) has 1-norm within a factor of $r$
of the 1-norm minimizing solution). In fact, \cite{FLPXjogo} demonstrated that the factor of $r$ is actually very pessimistic, with much better results obtained in practice.

\section{Numerical experiments}\label{sec:exp}



We constructed eleven test instances
of varying sizes for our numerical experiments using the Matlab function \texttt{sprand} to  randomly generate $m\times n$ dense matrices $A$ with  rank $r$, as described in \cite[\S 2.1]{FLPXjogo}. We used  \texttt{Gurobi} to solve  \ref{prob:B} and \ref{prob:barB} as linear programs, and \texttt{Mosek} to solve \ref{prob:barC} and  \ref{prob:barA}
as second-order cone programs; see \cite{mosek} and \cite{gurobi}.
We ran our experiments on a
16-core machine (running Windows Server 2016 Standard):
two Intel Xeon CPU E5-2667 v4 processors
running at 3.20GHz, with 8 cores each, and 128 GB of memory.

In Table \ref{tab:results} we compare results for problems \ref{prob:barC}, \ref{prob:barA}, \ref{prob:barB}\,, and for the local-search procedure (LS) proposed   in \cite{XFLPsiam} (see \S \ref{sec:ls}) . We also compare  \ref{prob:barB} to \ref{prob:B} to show how effective the problem reduction is.
We set a time limit of 5 hours for solving each instance.  In the first two columns of Table \ref{tab:results}, we identify, respectively, $m,n,r$ for each instance, and the procedure used to compute $H$. For each instance/procedure, we show the number of non-zero rows of $H$ (NZR), the zero-norm, the 1-norm, and the 2,1-norm  of $H$ (
 $\|H\|_0$, $\|H\|_1$,  and $\|H\|_{2,1}$),  and the elapsed time to solve the problem, or run the local-search procedure,  in seconds (Time). 
We consider $10^{-5}$ as the tolerance to distinguish non-zero elements. The symbol `*'  in column `Time', indicates   that the problem was not solved to optimality either because the time limit was reached or because we ran out of  memory. 


We could only solve \ref{prob:B}  for $n\leq 120$, while \ref{prob:barB} could be solved for all $n\leq 280$, and when both problems were solved, the latter was solved much faster, showing the impact of reformulating the problem in the reduced format.

We see that the 1-norm works well as a  surrogate for  sparsity, as the decrease in the 1-norm always leads to a decrease on the zero-norm, when comparing \ref{prob:barC} or \ref{prob:barA} to \ref{prob:barB} or \ref{prob:B}\,. 

We  see that, although we have distinct solutions for \ref{prob:barC} and \ref{prob:barA}, the difference between their 1-norms is small, showing that, for our instances, the minimization of the 2,1-norm in \ref{prob:barC} already leads to solutions with the 1-norm close to  minimum.  

We see that the 2,1-norm works well as a surrogate  to  the number of non-zero rows, as the decrease on the 2,1-norm always leads to a decrease on `NZR', when comparing \ref{prob:barB} or \ref{prob:B} to \ref{prob:barC} or \ref{prob:barA}. 

If  it is important to obtain $H$ with small  2,1-norm or even 1-norm, problem \ref{prob:barC} is a good choice, as it can be solved very efficiently, when compared to the other norm-minimization problems considered.

Finally, as pointed out in \cite{XFLPsiam},  the  local-search procedure is very fast compared to the solution of the norm-minimization problems, and leads to solutions with much smaller zero-norms and number of non-zero rows. If sparsity and block structure are the main goals, then we conclude from our experiments, that  the local-search procedure is the best option. However,  both the 2,1-norm and the 1-norm of the solutions obtained are not close to the minimum possible value. On the other hand, we know from \cite{XFLPsiam}, that the factor between the 1-norm of the local-search solution and the solution of  \ref{prob:B} is bounded by $r$. In \cite{XFLPsiam}, we showed that although this bound is achieved in the worst case, for the random instances used in those experiments,  the factor was much smaller than the bound. Here, we confirm this result on the bigger instances that we could solve, owing to  the reduction obtained on the reformulation of \ref{prob:B} as \ref{prob:barB}\,,  the factors computed for all instances with $n\leq 280$, is less than $1.6$, while the rank goes up to 70.

\begin{table}[!ht]
\tiny
\centering
\begin{tabular}{l|l|rrrrr}
$m,n,r$                          & \multicolumn{1}{c|}{Prob} & NZR & $\|H\|_0$ & $\|H\|_1$ & $\|H\|_{2,1}$ &  Time (sec) \\ \hline
\multirow{5}{*}{$40,20,10$}      & $\mathcal{P}_1^{2,1}$     & 15                      & 569                           & 62.409                        & 14.390                                                      & 0.01                           \\
                                 & $(\mathcal{P}_1^{2,1})^1$ & 15                      & 569                           & 62.389                        & 14.390                                                    & 0.14                           \\
                                 & $\mathcal{P}^1_{123}$     & 16                      & 542                           & 59.985                        & 14.960                                                    & 0.27                           \\
                                 & ${P}^1_{123}$             & 16                      & 542                           & 59.985                        & 14.960                                                      & 7.98                           \\
                                 & LS                        & 10                      & 380                           & 72.551                        & 16.393                                                       & 0.01                           \\ \hline
\multirow{5}{*}{$80,40,20$}      & $\mathcal{P}_1^{2,1}$     & 34  & 2584      & 181.656   & 29.324              & 0.02       \\
                                 & $(\mathcal{P}_1^{2,1})^1$ & 34  & 2584      & 181.606   & 29.324            & 2.11       \\
                                 & $\mathcal{P}^1_{123}$     & 35  & 2344      & 174.111   & 30.537             & 4.14       \\
                                 & ${P}^1_{123}$             & 35  & 2344      & 174.111   & 30.537             & 328.51     \\
                                 & LS                        & 20  & 1520      & 205.756   & 33.239           & 0.02       \\ \hline
\multirow{5}{*}{$120,60,30$}     & $\mathcal{P}_1^{2,1}$     & 49  & 5727      & 280.015   & 44.738             & 0.11       \\
                                 & $(\mathcal{P}_1^{2,1})^1$ & 49  & 5726      & 279.921   & 44.738               & 10.08      \\
                                 & $\mathcal{P}^1_{123}$     & 54  & 5518      & 256.148   & 48.145               & 11.69      \\
                                 & ${P}^1_{123}$             & 54  & 5518      & 256.148   & 48.145               & 12574.51   \\
                                 & LS                        & 30  & 3508      & 368.637   & 56.173              & 0.01       \\ \hline
\multirow{5}{*}{$160,80,40$}     & $\mathcal{P}_1^{2,1}$     & 62  & 9732      & 374.365   & 57.762               & 0.22       \\
                                 & $(\mathcal{P}_1^{2,1})^1$ & 62  & 9729      & 374.224   & 57.762               & 94.18      \\
                                 & $\mathcal{P}^1_{123}$     & 73  & 9824      & 338.068   & 61.635               & 75.08      \\
                                 & ${P}^1_{123}$             & -   & -         & -         & -                        & *          \\
                                 & LS                        & 40  & 6289      & 454.595   & 66.090              & 0.02       \\ \hline
\multirow{5}{*}{$200,100,50$}    & $\mathcal{P}_1^{2,1}$     & 75  & 14528     & 561.684   & 72.490               & 0.80       \\
                                 & $(\mathcal{P}_1^{2,1})^1$ & 75  & 14526     & 561.530   & 72.490               & 221.85     \\
                                 & $\mathcal{P}^1_{123}$     & 89  & 14918     & 516.202   & 78.013               & 344.54     \\
                                 & ${P}^1_{123}$             & -   & -         & -         & -                        & *          \\
                                 & LS                        & 50  & 9693      & 770.098   & 90.294              & 0.28       \\ \hline
\multirow{5}{*}{$240,120,60$}    & $\mathcal{P}_1^{2,1}$     & 102 & 23812     & 752.147   & 90.253               & 0.97       \\
                                 & $(\mathcal{P}_1^{2,1})^1$ & 102 & 23798     & 751.934   & 90.253               & 782.32     \\
                                 & $\mathcal{P}^1_{123}$     & 114 & 22591     & 678.332   & 97.611               & 269.50     \\
                                 & ${P}^1_{123}$             & -   & -         & -         & -                        & *          \\
                                 & LS                        & 60  & 14026     & 1069.745  & 115.824             & 0.05       \\ \hline
\multirow{5}{*}{$280,140,70$}    & $\mathcal{P}_1^{2,1}$     & 115 & 31108     & 921.040   & 104.272             & 8.57       \\
                                 & $(\mathcal{P}_1^{2,1})^1$ & 115 & 31100     & 920.799   & 104.272           & 6385.20    \\
                                 & $\mathcal{P}^1_{123}$     & 131 & 30278     & 837.155   & 112.825            & 13401.73   \\
                                 & ${P}^1_{123}$             & -   & -         & -         & -                      & *          \\
                                 & LS                        & 70  & 18959     & 1246.666  & 129.702             & 0.05       \\ \hline
\multirow{5}{*}{$320,160,80$}    & $\mathcal{P}_1^{2,1}$     & 132 & 41075     & 1058.841  & 119.115             & 8.50       \\
                                 & $(\mathcal{P}_1^{2,1})^1$ & -   & -         & -         & -                        & *          \\
                                 & $\mathcal{P}^1_{123}$     & -   & -         & -         & -                        & *          \\
                                 & ${P}^1_{123}$             & -   & -         & -         & -                        & *          \\
                                 & LS                        & 80  & 24939     & 1524.077  & 152.223             & 0.08       \\ \hline
\multirow{5}{*}{$1000,500,250$}  & $\mathcal{P}_1^{2,1}$     & 442 & 432276    & 5743.575  & 382.905             & 102.89     \\
                                 & $(\mathcal{P}_1^{2,1})^1$ & -   & -         & -         & -                        & *          \\
                                 & $\mathcal{P}^1_{123}$     & -   & -         & -         & -                        & *          \\
                                 & ${P}^1_{123}$             & -   & -         & -         & -                        & *          \\
                                 & LS                        & 250 & 245093    & 12088.928 & 626.678             & 11.41      \\ \hline
\multirow{5}{*}{$2000,1000,500$} & $\mathcal{P}_1^{2,1}$     & 913 & 1801923   & 14267.101 & 778.034             & 1507.02    \\
                                 & $(\mathcal{P}_1^{2,1})^1$ & -   & -         & -         & -                        & *          \\
                                 & $\mathcal{P}^1_{123}$     & -   & -         & -         & -                        & *          \\
                                 & ${P}^1_{123}$             & -   & -         & -         & -                        & *          \\
                                 & LS                        & 500 & 985950    & 33206.475 & 1301.181          & 63.40      \\ \hline
\multirow{5}{*}{$3000,1500,750$} & $\mathcal{P}_1^{2,1}$     & -   & -         & -         & -                       & *          \\
                                 & $(\mathcal{P}_1^{2,1})^1$ & -   & -         & -         & -                        & *          \\
                                 & $\mathcal{P}^1_{123}$     & -   & -         & -         & -                        & *          \\
                                 & ${P}^1_{123}$             & -   & -         & -         & -                        & *          \\
                                 & LS                        & 750 & 2224075   & 63486.782 & 2002.052            & 98.78      
\end{tabular}
\caption{Comparison between procedures}
\label{tab:results}
\end{table}

\section*{Acknowledgments} 
The authors gratefully acknowledge: (i) Laura Bolzano and Ahmad Mousavi
for suggesting to us  the use of norms like the  $2,1$-norm to attempt to induce structured sparsity of a generalized inverse, and (ii) Fei Wang for helpful conversations on our topic. 

 G. Ponte was supported in part by CNPq GM-GD scholarship 161501/2022-2. M. Fampa was supported in part by CNPq grants 305444/2019-0 and 434683/2018-3.  
J. Lee was supported in part by AFOSR grant FA9550-22-1-0172. 
This work is partially based upon work supported by the 
National Science Foundation under Grant No. DMS-1929284 while 
all of the authors were in residence at the Institute for Computational and Experimental Research in Mathematics (ICERM) in Providence, RI, during the Discrete Optimization program.

\bibliographystyle{alpha}
\bibliography{ginv}

\end{document}